\documentclass[]{article}

\usepackage{amsfonts}

\usepackage{amsmath}

\usepackage{amssymb}

\usepackage{amstext}

\usepackage{amsthm}

\usepackage{bbm}

\usepackage{bussproofs}

\usepackage{color}

\usepackage{enumitem}

\usepackage{epigraph}

\usepackage{footmisc}

\usepackage{graphicx}

\usepackage{mathbbol}

\usepackage[mathcal]{euscript}

\usepackage[colorlinks=true,citecolor=blue]{hyperref}

\usepackage{doi}

\usepackage{mathrsfs}

\usepackage{stmaryrd} 

\usepackage[all]{xy}
\xyoption{2cell}
\xyoption{curve}
\UseAllTwocells
\SelectTips{cm}{}

\usepackage{geometry}
\usepackage{url}

\theoremstyle{plain}
\newtheorem{theorem}{Theorem}[section]
\newtheorem{proposition}[theorem]{Proposition}
\newtheorem{corollary}[theorem]{Corollary}
\newtheorem{lemma}[theorem]{Lemma}

\theoremstyle{definition}
\newtheorem{definition}[theorem]{Definition}
\newtheorem{remark}[theorem]{Remark}
\newtheorem{example}[theorem]{Example}

\newtheorem{our-remark}[theorem]{Remark}

\newcommand{\V}{\mathcal V}

\def\vcat{\V\hspace{-1.5pt}\mbox{-}\hspace{.5pt}{\mathbf{Cat}}}

\def\vdist{\V\hspace{-0.5pt}\mbox{-}\hspace{.9pt}{\mathbf{Dist}}}

\newcommand{\vsup}{\V\hspace{0pt}\mbox{-}\hspace{.5pt}\mathbf{Sup}}
\newcommand{\vbisup}{\V\hspace{0pt}\mbox{-}\hspace{.5pt}\mathbf{BiSup}}

\newcommand{\vccd}{\V\hspace{0pt}\mbox{-}\hspace{.5pt}\mathbf{ccd}}

\def\kat#1{{\mathscr{#1}}}
\def\X{\kat{X}}
\def\Y{\kat{Y}}
\def\Z{\kat{Z}}

\def\A{\kat{A}}
\def\B{\kat{B}}
\def\C{\kat{C}}

\newcommand{\set}{\mathbf{Set}}

\newcommand{\ord}{\mathbf{Ord}}

\newcommand{\twosup}{\mathbf{Sup}}

\newcommand{\two}{\mathbbm 2}
\newcommand{\one}{\mathbbm 1}

\def\twocat{\two\hspace{-0.5pt}\mbox{-}\hspace{.5pt}{\mathbf{Cat}}}

\def\colim#1#2{{\mathsf{colim}_{#1}{#2}}}

\def\sup#1#2{{\mathsf{sup}_{#1}{#2}}}

\newcommand{\id}{\mathsf{id}}

\newcommand{\op}{^{\mathsf{op}}}

\newcommand{\sep}{_{\mathsf{sep}}}

\newcommand{\Lan}{\mathsf{Lan}}

\renewcommand{\d}{\mathsf{d}} 

\newcommand{\D}{\mathbbm D}

\newcommand{\U}{\mathbbm U}

\newcommand{\y}{\mathsf{y}}

\renewcommand{\t}{\mathsf{t}}

\newcommand{\ot}{\otimes}
\def\td{\ot_\V}

\newcommand{\nto}{\nrightarrow}

\renewcommand{\phi}{\varphi}

\begin{document}

\title{On the tensor product of completely distributive quantale-enriched categories}
                       
\author{Adriana Balan\thanks{Department of Mathematical Methods and Models, Fundamental Sciences Applied in Engineering Research Center, National University of Science and Technology POLITEHNICA of Bucharest, RO-060042, Bucharest, Romania. \linebreak {\tt adriana.balan@upb.ro}}\thanks{The author acknowledges support by the research project 88/11.10.2023 GNaC2023 ARUT.}}
%
%

\maketitle      


\begin{abstract}
Tensor products are ubiquitous in algebra, topology, logic and category theory. 
The present paper explores the monoidal structure of the category $\vsup$ of separated cocomplete enriched categories over a commutative quantale $\V$, the many-valued analogue of complete sup-lattices.  
We recover the known result that $\vsup$ is $*$-autonomous and we show that the nuclear/dualizable objects in $\vsup$ are precisely the completely distributive cocomplete $\V$-categories. 

\end{abstract}

\section{Introduction}


Quantitative reasoning is often encountered in Computer Science, in various contexts and guises: 
in the theory of dynamical systems featuring commensurable data, such as probabilistic transition systems~\cite{BreugelWorell2005}, 
or via behavioural distances substituting program equivalences~\cite{MardarePanangadenPlotkin2016}, in Formal Concept Analysis~\cite{Belohlavek2004}, and the examples could go further on. 
Their common feature is the semantic interpretation in domains having a metric-like structure -- that is, quantales, which generalise both truth values and distances, resulting in a rich categorical structure: cocomplete quantale enriched categories~\cite{Stubbe2005}. 
In this paper we shall bring forward their symmetric tensor product~\cite{JoyalTierney1984}, in a presentation featuring elements of enriched category theory, Kock-Z\"oberlein monads, Galois connections and Shmuely's G-ideals. We shall also relate dualizability and complete distributivity in the category of separated cocomplete quantale-enriched categories.

\medskip

The tensor product of complete sup-lattices has a long history (see for example~\cite{BanaschewskiNelson1976,Mowat1968,Shmuely74} and references therein). Essentially, the theory developed along two main lines of research: first, the description of the tensor product as Galois maps, and second, as certain downsets separately preserving suprema (called G-ideals in~\cite{Shmuely74}). Each description allowed for various generalisations to arbitrary posets, classes of lattices, closure spaces or formal contexts. 

\medskip

Quantales were introduced in the '80s as a logical-theoretic framework for studying certain spaces arising from quantum mechanics~\cite{Mulvey1986}. 
As mentioned above, these can be perceived both as lattices of ``truth-values'' and as ``distances'', equipped with an extra operation expressing conjunction (logical interpretation) or addition of distances (metric interpretation). 

\medskip

Lawvere observed that quantales provided a unified setting for both ordered sets and metric spaces as enriched categories~\cite{Lawvere73}. His insight significantly enhanced the  quantitative theory of domains~\cite{Stubbe2007a,Wagner1994} mentioned in the beginning with  concepts and ideas from category theory. Among the most useful and pleasing properties of (quantale-)enriched categories are undoubtedly (co)completeness with respect to certain classes of (co)limits and commutation of such limits and colimits (with prominent examples featuring the ordered case, like frames, continuous lattices or completely distributive lattices). 
For $\V$ a commutative quantale, the category $\vsup$ of separated cocomplete quantale-enriched categories and cocontinuous functors~\cite{JoyalTierney1984,PedicchioTholen89,Stubbe2005} generalises complete sup-lattices to the many-valued realm. 
That $\vsup$ is a $*$-autonomous category with tensor product $\A\td \B$ classifying bimorphisms ($\V$-functors $\A \otimes\B \to \C$ cocontinuous in each variable separately) is an old result, going back to~\cite{JoyalTierney1984}, which has recently seen a revival of interest~\cite{EklundGutierrez-GarciaHohleKortelainen2018,Tholen2024}. 
In the former reference, the classical construction of the tensor product of modules over a commutative ring is applied {\em mutatis mutandis} to modules over a monoid in the category of complete sup-lattices, that is, over a quantale. 
In the latter two references, the emphasis is on the underlying category of complete sup-lattices, respectively on the equational presentation of cocomplete enriched categories as endowed with an action of the quantale. 

\medskip

The present paper contributes to the above mentioned results with yet another description of the tensor product, relying on quantale-enriched categories techniques: KZ-monads, enriched limits and colimits, left Kan extensions are the main tools which build the path for our results. 
Although certainly more abstract than the previous cited papers, the categorical approach has the advantage that it can suit also other purposes, like generalization to quantaloid-enriched categories, dropping thus the commutativity assumption on the quantale.  

\medskip

Section~\ref{sec:Vsup} of the paper was influenced by the results of~\cite{RosebrughWood1994}. 

\paragraph{Structure of the paper.} 
In Section~\ref{sec:prelim} we recall the elements of quantale-enriched category theory, distributors, colimits and cocompleteness as used in this paper. 
The next section is devoted to the free cocompletion monad, where we recollect its main features to be used subsequently. 
Sections~\ref{sec:Vsup} and~\ref{sec:Vccd} contain the main results: the construction and properties of the tensor product of separated cocomplete quantale-enriched categories, that the tensor product restricts to the full subcategory of completely distributive cocomplete quantale-enriched categories, and that the objects of the latter category are precisely the nuclear/dualizable cocomplete quantale-enriched categories.


\section{Preliminaires}\label{sec:prelim}

Here we recall the main elements of quantale-enriched category theory that we shall use in the sequel. 
For more details, we point to~\cite{HofmannSealTholen2014,kelly:book,Lawvere73,Stubbe2005} and references therein.


\paragraph{Quantales.}
A {\em (commutative) quantale} is a (commutative) monoid $(\V, \ot, e)$ in the category $\mathbf{Sup}$ of complete sup-lattices and sup-preserving morphisms. 
In particular, $\V$ carries a complete partial order $\le$. 
Because the multiplication ({\em tensor product}) of the quantale preserves suprema in each variable, every $v \otimes - :\V \to \V$ has a right adjoint $[v,-]:\V \to \V$ ({\em internal hom}). 
Adjointness means that
\[
u \otimes v \leq w \quad \iff \quad u \leq [v,w] 
\]
holds for each $u,v,w\in \V$. 
Hence a commutative quantale is a posetal symmetric monoidal closed category, complete and cocomplete. 

\medskip

We list below several examples of quantales:
\begin{example}\label{ex:quantales}
\begin{enumerate}
\item\label{ex:2-quantale} 
The simplest example of a (commutative) quantale is the two-element chain $\V = (\two=\{0,1\}, \land,1)$, with meet as multiplication. 
\item\label{ex:3-quantale} The three-element chain $\mathbb 3=\{0<a<1\}$ supports two quantale structures for which the tensor is idempotent~\cite{CasleyCrewMeseguerPratt91}: 
\begin{itemize}
\item Taking $\ot$ to be the meet, one obtains the usual Heyting algebra $\V=(\mathbb{3}, \land, 1)$.
\item The second idempotent multiplication $\otimes$ on $\mathbb 3$ has unit $a$. 
The resulting quantale $(\mathbb 3, \ot,a)$ is a {\em Sugihara monoid}~\cite{OlsonRaftery07}. 
\end{itemize}
Besides the two idempotent structures mentioned above, there exists only one more quantale structure on $\mathbb 3$ (non-idempotent and integral), namely the {\em {\L}ukasiewicz} tensor product $v \otimes w = \max(0, v+w-1)$. 
\item\label{ex:[0,infty]-quantale}
The (extended) positive real numbers $([0, \infty], \ge)$ form a commutative quantale when endowed with addition and zero. The internal hom is $[u,v] = \max(v-u,0)$.
\item\label{ex:[0,1]-min-quantale} The unit interval $([0,1],\le)$, with tensor product $u\otimes v=\min(u,v)$ and internal hom $[u,v]=1 \textsf{ if }u\le v \textsf{ else } v$ is another example of a commutative quantale.
\item The left continuous distribution functions $\Delta =\{f:[0,\infty]\to [0,1] \mid f(a) = \bigvee_{b<a}f(b)\}$ form a commutative quantale with pointwise order and suprema (but not infima, which are obtained as $\left(\bigwedge_i f_i\right)(x) = \sup{y<x}{\mathsf{inf}_i f_i(y)}$)~\cite{HofmannSealTholen2014}. 
The tensor product is given by convolution, with unit the delta distribution at $0$. 
This quantale is in fact the {\em coproduct} of $[0,\infty]$ and $[0,1]$ in the category of quantales\cite{GutierrezGarciaHohleKubiak2017}. 
\item\label{ex:power-quantale} The powerset $\mathcal P(M)$ of a a (commutative) monoid $M$, with concatenation as $\ot$ and empty string as unit, is a (commutative) quantale; in particular, the powerset $\mathcal P(A^*)$ of the list monoid $A^*$ over a set $A$ is the free quantale over $A$.
\end{enumerate}
\end{example}
In the sequel, all quantales will be assumed commutative. 


\paragraph{Quantale enriched categories.}

Let $(\V, \ot, e,[-,-])$ be a commutative quantale.
A {\em $\V$-category} $\X$ consists of a set, still denoted $\X$, together with a map%
\footnote{
In the sequel we shall refer to this map as the $\V$-hom, $\V$-metric or $\V$-distance.
%
By slightly abuse, both the underlying set (of objects) and the $\V$-hom will be denoted by the same letter as the $\V$-category in question.} %
$\X:\X\times \X \to \V$ 
satisfying 
\[
e\leq \X(x,x)
\qquad \mbox{ and } \qquad 
\X(x,x')\otimes \X(x',x'') \leq \X(x,x'')
\]
for all $x,x',x''\in \X$. 
Each $\V$-category $\X$ carries an underlying order (reflexive and transitive relation) given by $x\le x' \Leftrightarrow e\le \X(x,x')$. 
A $\V$-category will be called {\em separated} (or {\em skeletal}) if its underlying order is in fact a partial order. 
The {\em opposite} of a $\V$-category $\X$, denoted $\X\op$, has the same objects as $\X$, and $\V$-homs $\X\op(x,x')= \X(x',x)$. 
A {\em $\V$-functor} $f:\X \to \Y$ is a map between the underlying sets satisfying $\X(x,x')\leq \Y(f(x),f(x'))$ for every objects $x,x'$ of $\X$. In particular, a $\V$-functor is monotone with respect to the underlying orders.
Finally, a {\em $\V$-natural transformation} $f\to g:\X \to \Y$ only accounts for pointwise inequality $e\leq \Y(f(x),g(x))$, $\forall \ x \in \X$. 

Notice that we use the same symbol $\leq$ for the order relation in both $\V$ and $\V$-categories, and rely on the context to tell them apart.
We shall proceed similarly for joins and meets with respect to the underlying order of an $\V$-category, whenever these exist.

\begin{example}\label{ex:vcats}
\begin{enumerate}
\item\label{ex:V as a V-category} 
The quantale $\V$ becomes itself a $\V$-category with $\V(v,w) = [v,w]$. 
The induced order is in fact the underlying order of the quantale. 
Being antisymmetric, $\V$ is separated as a $\V$-category. 
\item\label{ex:discr-vcat} 
Each set $A$ can be perceived as a (separated) $\V$-category $\mathsf{d}A$ when equipped with $\mathsf{d}A(a,b) = e\textsf{ if }a=b \textsf{ else } \bot$.  
Such $\V$-categories are called \emph{discrete}. 
\item \label{ex:lawvere} Ordered sets $(A,\leq)$ are enriched categories over the two-element quantale $\V=\mathbb 2$~\cite{Lawvere73}. 
The monotone maps are precisely the $\V$-functors. 
\item If $\V$ is the extended real half line $([0,\infty]\op, +, 0)$ as in Example~\ref{ex:quantales}.\ref{ex:[0,infty]-quantale}, a small $\V$-category $\X$ is a  generalised metric space~\cite{Lawvere73}: a set $\X$ endowed with a mapping $\X: \X \times \X \to [0, \infty]$ satisfying
\[
0\geq \X(x,x), \quad \X(x,x') + \X(x',x'') \geq \X(x,x'') \ \ .
\]
Observe that $\X: \X \times \X \to [0, \infty]$ is only a pseudo-metric, in the sense that two distinct points may have distance $0$, distances are not necessarily symmetric and are allowed to be infinite. 
An important consequence of the latter property is that the category of generalised metric spaces has colimits (whereas metric spaces do not even have coproducts).
A $\V$-functor is a non-expanding map. 
\end{enumerate}	
\end{example}

Denote by $\vcat$ the 2-category of (small) $\V$-enriched categories, $\V$-functors and $\V$-natural transformations, and by $\vcat\sep$ the full 2-subcategory consisting of  separated $\V$-categories.  
In fact, both $\vcat$ and $\vcat\sep$ are locally ordered $2$-categories, a result which will considerably simplify our reasoning.
Recall that the forgetful functor $\vcat \to \set$ is topological, and that the embedding $\vcat\sep \to \vcat$ is strongly epi-reflective~\cite{HofmannSealTholen2014}. 
Consequently, both $\vcat$ and $\vcat\sep$ are complete and cocomplete as ordinary categories. 
$\vcat$ is also symmetric monoidal closed. 
The tensor product $\X \otimes\Y$ of two $\V$-categories $\X$ and $\Y$ has as objects pairs $(x,y)$ of objects $x$ of $\X$ and $y$ of $\Y$, and $\V$-homs 
\[
(\X\otimes\Y)((x,y),(x',y'))= \X(x,x')\otimes\Y(y,y')
\]
Notice that we use the same symbol for the tensor product of $\V$-categories and for the underlying multiplication of the quantale, relying on the context to distinguish them. 
\begin{remark} If the quantale $\V$ is integral (that is, the unit $e$ coincides with the top element of the quantale), then the tensor product of two separated quantale-enriched categories is again separated. 
However, this property is not guaranteed in general. To see an example, consider the $4$-element Boolean algebra $\V = \{\bot,e,a,\top\}$, with commutative quantale structure given by $a\otimes a =e, a \otimes \top = \top$. 
The elements $\top,e,\bot$ are idempotent and the unit is $e$. 
This quantale is denoted $R^{4,2}_{2,2}$ in~\cite{GalatosJipsen}. 
Taking $\X=\Y=\V$, one has $\X\otimes \Y((a,e),(e,a)) = [a,e]\otimes [e,a] = a \otimes a = e = \X\otimes \Y((e,a),(a,e))$, but $(e,a)\neq (a,e)$. 
\end{remark}
The unit for the tensor product of $\V$-categories is the $\V$-category $\one$ with one object, denoted $0$, and $\V$-hom $\one(0,0)=e$. 
The internal $\V$-category-hom $[\X, \Y]$ has as objects $\V$-functors $f:\X \to \Y$, with $\V$-valued hom $[\X, \Y](f,f') = \bigwedge_x \Y(fx,f'x)$.


\paragraph{Quantale-enriched distributors.}

This is only a brief overview, for more details the reader is invited to \linebreak consult~\cite{Stubbe2005,Stubbe2006}.

\medskip

A {\em $\V$-distributor} (or profunctor, or (bi)module) $\varphi:\X \nrightarrow \Y$ between $\V$-categories is a $\V$-functor $\varphi:\Y\op \otimes\X \to \V$. 
Explicitly, 
\[
\Y(y',y)\otimes\varphi(y,x)\otimes\X(x,x') \leq \varphi (y',x')
\]
holds for all $x,x'\in \X$ and $y,y'\in \Y$. 
Composition of $\V$-distributors $\phi:\X \nrightarrow \Y$, $\psi:\Y \nrightarrow \Z$ is given by ``matrix multiplication''
\[
(\psi \otimes \phi)(z,x) = \bigvee_{y\in \Y} \psi(z,y) \otimes\phi(y,x)
\] %
This composition is strictly associative, due to enrichment in a quantale (whose underlying order is antisymmetric). The identity distributor on a $\V$-category $\X$ is the $\V$-hom $\X(-,-)$; in particular, the relations  
\[
\phi \otimes \X = \phi \ \mbox{ and } \ \Y \otimes \phi = \phi\]
hold for any $\V$-distributor $\phi:\X \nrightarrow \Y$. 
When $\phi$ is a contravariant, respectively covariant presheaf (see below), these relations are referred as the {\em co-Yoneda lemma}. 
Distributor composition admits both right extensions and right liftings
\[
\begin{array}{c}
\psi \otimes \phi \leq \xi 
\quad \Longleftrightarrow \quad 
\psi \leq \xi \swarrow \phi 
\quad \Longleftrightarrow \quad 
\phi \leq \psi \searrow \xi
\end{array} 
\]
where 
\[
(\xi\swarrow \phi)({{z}},{{y}})  = \bigwedge_{{x\in \X}} [\phi({{y}},{{x}}),\xi({{z}},{{x}})]
\ \mbox{ and } \ 
(\psi\searrow\xi) ({{y}},{{x}})  = \bigwedge_{{z\in \Z}} [\psi({{z}},{{y}}),\xi({{z}},{{x}})]
\]
for $\phi:\X \nrightarrow \Y$, $\psi:\Y \nrightarrow \Z$, $\xi:\X \nrightarrow \Z$.

Consequently, $\V$-categories, $\V$-distributors and $\V$-natural transformations (that is, pointwise inequalities) between them form a biclosed bicategory $\vdist$ (in fact a locally ordered $2$-category, given the enrichment in a quantale). 
A pair of distributors $\phi :\X \nrightarrow \Y$ and $ \psi:\Y \nrightarrow \X$ are {\em adjoints}, written $\phi \dashv \psi$, if $\X \le \psi \otimes \phi$ and $\phi \otimes \psi \le \Y$ hold. 
In particular, each $\V$-functor $f:\X\to \Y$ induces a pair of adjoint distributors $f_*\dashv f^*$ between $\X$ and $\Y$, where $f_*(y,x) = \Y(y,fx)$ and $f^*(x,y) = \Y(fx,y)$. 
In particular, observe that a $\V$-functor $f:\X \to \Y$ is left adjoint in $\vcat$, with $g:\Y \to \X$ as its right adjoint, if and only if $f^* = g_*$. 

\medskip

A distributor $\phi:\one \nto \X$ (equivalently, a $\V$-functor $\phi:\X \op \to \V$) is usually called {\em a contravariant presheaf} in category theory. It can be perceived as an $\V$-valued downset:  
the relation $\X\op(x',x) \leq [\phi(x'), \phi(x)]$, equivalent to $\X(x,x') \otimes\phi(x') \leq \phi(x)$, reads in case $\V=\two$ as 
\[
\left(x\leq x' \mbox{ and } x'\in \phi \right) \mbox{ implies } x\in \phi
\]
that is, $\phi$ is a downset in the usual sense. 
Here we implicitly identified a downset with its associated characteristic function. 
To preserve this intuition, we shall denote by $\D\X$ the $\V$-category of contravariant presheaves $\vdist(\one, \X) = [\X\op, \V]$. 
One of its prominent features that we recall is that it classifies distributors into $\X$~\cite{Stubbe2005}: there is an isomorphism of $\V$-categories
\begin{equation}\label{eq:classif-dist}
\vdist(\Y, \X) \cong \vcat(\Y, \D\X)
\end{equation}
functorial in the $\V$-categories $\X,\Y$, mapping a $\V$-distributor $\phi:\Y \nto \X$ to the $\V$-functor $f_\phi(y) = \phi(-,y)$, respectively a $\V$-functor $f:\Y \to \D \X$ to the $\V$-distributor $\phi_f(x,y) = f(y)(x)$.

Dually, a {\em covariant presheaf} is a $\V$-distributor $\psi: \X \nto \one$ (that is, $\psi:\X \to \V$), the $\V$-valued analogue of an upper set, and $\U\X = [\X, \V]\op$ is the $\V$-category of covariant presheaves (notice the ``$\mathsf{op}$''!).

We end this paragraph on distributors by recalling the {\em Yoneda embedding}: the fully faithful $\V$-functor $\y_\X:\X\to [\X\op, \V] = \D\X$ mapping an object $x$ of $\X$ to the {representable} contravariant presheaf $\X(-, x)$, corresponding, under the equivalence in~\eqref{eq:classif-dist}, to the identity distributor on $\X$. 
For $\V=\two$, this is the familiar principal downset embedding of an ordered set. 

For every contravariant presheaf $\phi:\one \nto \X$, the following relation hold ({\em Yoneda lemma}):
\[
\D\X(\y_\X(-),\phi) = \phi
\]
For completeness, we also mention the covariant version of the Yoneda embedding, namely $\y^\prime_\X: \X \to [\X,\V]\op = \U\X$, $\y^\prime_\X (x) = \X(x,-)$ which satisfies 
\[
[\X,\V](\y^\prime_\X(-),\psi) = \psi
\] 
for every covariant presheaf $\psi: \X \nto \one$. 


\paragraph{Limits and colimits.} 
Given a distributor $\xymatrix{\X\ar[r]|-@{|}^\phi & \Y}$ and a $\V$-functor $f:\Y \to \Z$, the {\em colimit} of $f$ weighted by $\phi$ is a $\V$-functor $\colim{\phi}{f}:\X \to \Z$ representing the distributor $\phi \searrow f^*$, i.e. 
\begin{equation}\label{eq:colim}
(\colim{\phi}{f})^* \cong \phi \searrow f^* 
\qquad \xymatrix{\X \ar[rr]|-@{|}^\phi \ar@{.>}[dr]|-@{}_{\colim{\phi}{f}} &&\Y \ar[dl]|-@{}^{f}
\\
& \Z   
}
\qquad \xymatrix{\X \ar[rr]|-@{|}^\phi \ar@{<.}[dr]|-@{|}_{(\colim{\phi}{f})^*} &&\Y \ar@{<-}[dl]|-@{|}^{f^*}
\\
& \Z  \ar@{}[u]|>>>>{\rightarrow} & }
\end{equation}
Equivalently, 
\[
\Z(\colim{\phi}{f}(x),z) \cong [\Y\op, \V](\phi(-,x), \Z(f-,z))
\]
Limits are defined dually. 
Some particular colimits shall be of interest in the sequel. We borrow their presentation from~\cite{Riehl2009}:

\begin{example}
\begin{enumerate}
	\item If the weight is the identity distributor on $\Y$, the colimit of $f:\Y \to \Z$ weighted by $\Y$ is precisely $f$:
	\[
	(\colim{\Y}{f})^* \cong \Y(-,-) \searrow f^* = f^*
	\]
    \item If the weight is the {\em left adjoint distributor} $\xymatrix{\X \ar[r]|-@{|}^{j_*} & \Y}$ associated to a $\V$-functor $\xymatrix{\X \ar[r]^{j} & \Y}$, then the above colimit satisfies 
    \[
    (\colim{j^*}{f})^* \cong j_* \searrow f^* = j^* \otimes f^* = (f\circ j)^* 
    \xymatrix{}
    \]
    hence the $j_*$-weighted colimit is given by {\em precomposition} with $j$. 
    
    \item Now, if the weight $\phi$ is instead the {\em right adjoint distributor} $\xymatrix{\X \ar[r]|-@{|}^{j^*} & \Y}$ associated to a $\V$-functor $\xymatrix{\Y \ar[r]^{j} & \X}$, then the relation 
    \begin{equation}\label{eq:pointwise lan}
    \Z(\colim{j^*}{f}(x),z) \cong [\Y\op, \V](\X(j-,x), \Z(f-,z))
    \end{equation}
   exhibits $\colim{j^*}{f}$ as the {\em pointwise left Kan extension} $\Lan_j f$ of $f$ along $j$. 
   
    \item If the weight is a contravariant presheaf $\phi:\one \nrightarrow \Y$, then $\colim{\phi}{f}:\one \to \Z$ picks an object of $\Z$ satisfying     \[
    \Z(\colim{\phi}{f},z) \cong [\Y\op, \V](\phi(-), \Z(f-,z))
    \]
that is, it becomes the usual $\V$-enriched colimit~\cite{kelly:book}. 
To be more precise, the isomorphism above is in fact an equality, due to the order on $\V$ being antisymmetric.    

    \item If both the domain and the codomain of the weight are the one-object $\V$-category $\one$, then the weight just picks an element $v \in \V$, and so does the functor -- it picks an object $z \in \Z$. The resulting colimit is known as the {\em tensor} of $z$ by $v$ and denoted $v \otimes z$. Explicitly, $v \otimes z$ is characterised by
    \[
    \Z(v\otimes z,-) \cong [v,\Z(z,-)]
    \]
    A $\V$-category having all tensors is called {\em tensored}. The use of the same symbol for the tensor as a colimit and the underlying multiplication of the quantale is motivated by the fact that in $\V$, seen as a $\V$-category, these two notions coincide. 
    \item 
    Let $K$ be a set, seen as a discrete $\V$-category and $\phi:\one \nto K$ the contravariant presheaf assigning the value $e$, the unit of $\V$, to each $k\in K$. To give a $\V$-functor $f:K \to \Z$ is the same as giving a family $(z_k)_{k\in K}$ of objects of $\Z$. Then the resulting colimit in $\Z$, denoted $\bigvee_k z_k$,
is called {\em the join} of the family $(z_k)_{k\in K}$ and verifies
    \[
    \Z(\bigvee_k z_k,-) \cong \bigwedge_k\Z(z_k,-)
    \]  
    In particular, the above colimit is also the {\em join} (hence the name) of the family $(z_k)_{k\in K}$ in the underlying ordered set $(\Z,\le)$, but the converse is not necessarily true, unless $\Z$ is {\em cotensored}~\cite{Stubbe2006} (that is, $\Z\op$ is tensored). 
     
\end{enumerate}
\end{example}

\paragraph{Cocomplete $\V$-categories.} As usual, a $\V$-category admitting all colimits will be called {\em cocomplete}, and a $\V$-functor preserving all colimits which exist in its domain, {\em cocontinuous}.

\begin{example} 
(Contravariant) presheaf categories are cocomplete: the colimit of a $\V$-functor $f:\Y \to \D \Z$ weighted by $\phi:\X \nto \Y$ is the $\V$-functor $\X \to \D \Z$ corresponding to the distributor $\phi_f \otimes \phi$, where $\phi_f$ is the distributor $\Y \nto \Z$ corresponding to $f$ under the equivalence~\eqref{eq:classif-dist}~\cite{Stubbe2005}. 
\end{example}

The colimit of a $\V$-functor $f:\Y\to \Z$ weighted by a distributor $\phi:\X \nto \Y$ can be expressed by tensors and joins, whenever these exist in the codomain $\V$-category $\Z$~\cite{Stubbe2006}:
\begin{equation}\label{eq:colim-join-tensor}
\left( \colim{\phi}{f} \right)(x) = \bigvee_y \phi(y,x) \ot f(y)
\end{equation}
In particular, the left Kan extension of $f:\Y \to \Z$ along $j:\Y \to \X$ can be explicitly written as
\[
\mathsf{Lan}_j f(x) = \bigvee_y \X ( j (y), x) \ot  f (y) 
\]

\begin{proposition}~\cite{Stubbe2005}\label{prop:cocts}
The following are equivalent for a $\V$-category $\X$:
\begin{enumerate}
\item $\X$ is cocomplete.
\item \label{prop:sup} $\X$ has all colimits of the identity functor weighted by contravariant presheaves: for each $\phi:\one \nto \X$, there is an object $\colim{\phi}{\id_\X}$ in $\X$ such that 
\begin{equation}\label{eq:sup}
\X(\colim{\phi}{\id_\X},x) = \D\X(\phi,\X(-,x)) = \bigwedge_{x'\in \X} [\phi(x'),\X(x',x)]
\end{equation}
holds for every object $x$ of $\X$. 
\item The Yoneda $\V$-functor $\y_\X:
\X \to \D\X$ is a right adjoint.
\end{enumerate}
\end{proposition}

Interpreting Proposition~\ref{prop:cocts}, item \ref{prop:sup} above in the ordered case ($\V=\two$), we see that $\colim{\phi}{\id_\X} \leq x$ holds if and only if $x' \in \phi $ implies $x' \leq x$ for all $x'$. 
Hence, intuitively, $\colim{\phi}{\id_\X}$ computes the {\em $\V$-supremum} of the $\V$-contravariant presheaf $\phi:\X\op \to \V$ (thinking again of $\phi$ as a ``$\V$-valued downset'') and will be denoted $\sup{\X}{\phi}$. 
Then cocompleteness of a $\V$-category $\X$ can be rephrased as the existence of all $\V$-suprema in $\X$ and Equation~\eqref{eq:sup} shows that the left adjoint of $\y_\X$ is precisely $\sup{\X}{}$. 
As $\y_\X$ is fully-faithful, $\sup{\X}{}\circ \y_\X \cong \id_\X$ holds. 
Lastly, express the $\V$-supremum of a contravariant presheaf $\phi:\X\op \to \V$ using joins and tensors using~\eqref{eq:colim-join-tensor}:
\begin{equation}\label{eq:sup-as-join-of-tensor}
\sup{\X}{\phi} = \bigvee_{x\in \X} \phi(x) \ot x
\end{equation}

\medskip

There is yet another description of cocomplete $\V$-categories that we recall for completeness, although we shall not need it: under the extra assumption of being separated, the cocomplete $\V$-categories are precisely the $\V$-modules in $\mathbf{Sup}$. That is, they are sup-lattices (with respect to the underlying order) endowed with an action of the quantale~\cite{PedicchioTholen89,Stubbe2006}. 

\medskip

In view of the above, we shall denote by $\vsup$ the  ($2$-)category of cocomplete $\V$-categories and cocontinuous $\V$-functors, and by $\vsup\sep$ its full sub-$2$-category consisting of separated cocomplete $\V$-categories. It is easy to see that a cocomplete $\V$-category is separated if and only if the isomorphism $\sup{\X}{}\circ \y_\X \cong \id_\X$ is in fact an equality, and that $\vsup\sep$ is biequivalent to $\vsup$~\cite{Stubbe2006}.


\section{The free cocompletion monad on $\vcat$}


In this section we recall the free cocompletion monad $\D$ on $\V$-categories, also known as the (contravariant) presheaf monad. 
References include~\cite{Stubbe2005,Stubbe2007a,Stubbe2010}. 
In general enriched category theory $\D$ is only a pseudomonad, but here we benefit from enriching in a quantale and obtain a genuine $2$-monad. 
Besides that, all quantale-enriched categories are small, including their free cocompletions, so we do not have to worry about size issues. 
The novelty of our presentation is the emphasis on the monoidal structure of $\D$, and the results obtained from this approach.


\paragraph{The ($2$-)functor $\D$.}
The functor $\D:\vcat \to \vcat$ maps a $\V$-category $\X$ to the (separated) $\V$-category of contravariant presheaves $\D\X = [\X\op,\V]$, and a $\V$-functor $f:\X \to \Y$ to the left Kan extension $\D f = 
\mathsf{Lan}_{\y_\X}%
(\y_\Y \circ f)
:\D\X \to \D\Y$:
\[
\xymatrix{
\X \ar[r]^{\y_\X} \ar[d]_f \ar@{}[dr]|{\to}
& 
\D\X \ar@{.>}[d]^{\D f}
\\
\Y \ar[r]^{\y_\Y} 
&
\D\Y
}
\]
Explicitly, $\D f (\phi) = f_* \otimes \phi $. 
In particular, the $2$-cell (inequality) above expressing the unit of the left Kan extension is actually an equality, due to the fully faithfulness of the Yoneda embedding and to the presheaf category being separated. 
$\D$ is a locally fully faithful $2$-functor: for every $f,g:\X \to \Y$, the relation
\[
\begin{array}{lllll}
[\X,\Y](f,g) 
& = & 
[\X, \D \Y ](\y_\Y \circ f, \y_\Y \circ g) 
& = & 
[\X, \D \Y] (\y_\X \circ f, \D g \circ \y_\X) 
\\[10pt]
& = & 
[\D \X, \D \Y] (\D f, \D g ) 
\end{array}
\]
holds.

The formula $\D f = f_* \otimes -$ giving the action of $\D$ on arrows immediately shows that $\D f$ is part of a triple adjunction~\cite[Proposition~3.1]{Stubbe2007a}
\begin{equation}\label{eq:Df:triple-adj}
\xymatrix@C=60pt{
\D \X 
\ar@/^2.75ex/[r]^-{\D f}
\ar@{}@<1.5ex>[r]|{\perp}
\ar@{<-}[r]|{\D_{-1} f}
\ar@{}@<-1.5ex>[r]|{\perp}
\ar@/_2.75ex/[r]_-{\D_{\forall } f}
&
\D\Y
}
\end{equation}
with $\D_{-1} f = f^* \otimes - = f_*\searrow -$ and $\D_\forall f = f^*\searrow-$.\footnote{The notations $\D_{-1}$ and $\D_\forall$ were chosen as to remind of the inverse image, respectively universal image functors between ordinary powersets.}

\begin{remark}\label{rem:inverter}
For a fully faithful $\V$-functor $f:\X \to \Y$, the triple adjunction mentioned above
\[
\D f \dashv \D_{-1} f \dashv \D_\forall f
\]
is also fully faithful~\cite{DyckhoffTholen1987,KellyLawvere1989}, producing thus a {\em Unity and Identity of Adjointly Opposites (UIAO)}~\cite{Lawvere96}. 
%
As in any UIAO, there is an induced $2$-cell (inequality)
$\D f \le \D_\forall f $. 
We can thus consider the inverter~\cite[Proposition~2.3]{KenneyWood2010} in the $2$-category $\vcat$ of this $2$-cell, namely
\begin{equation}\label{eq:UIAO-inverter}
\xymatrix@C=40pt{ 
\mathsf{Inv}(\D f,\D_\forall f) \ar[r] & \D\X \ar@/^1.5ex/[r]^{\D f} \ar@{}[r]|{\downarrow} \ar@/_1.5ex/[r]_{\D_\forall f} & \D\Y} 
\end{equation}
The description is trivial: $\mathsf{Inv}(\D f,\D_\forall f)$ is the $\V$-subcategory of $\D\X$ consisting of every $\xymatrix@1{\one \ar[r]|-@{|}^\varphi & \X}$ such that $\D_\forall f (\varphi) \le \D f(\varphi)$ holds. 
However, using the 
covariant Yoneda embedding on $\D\Y$, we can rewrite $\D_\forall f (\varphi) \le \D f(\varphi)$ for latter use as 
\[
\begin{array}{lll}
\D\X(\psi \circ f, \phi) 
&
\le 
&
\D\Y(\psi, \D f(\phi))
\end{array}
\]
for every covariant presheaf $\psi:\Y \nto \one$. 
Notice that the reverse inequality always holds, so in fact in the above we have equality (given that both sides evaluate in $\V$).  

In the particular case when the fully faithful $f:\X \to \Y$ is the Yoneda embedding $\y_\X:\X \to \D\X$, the inverter $\mathsf{Inv}(\D \y_\X,\D_\forall \y_\X)$ in~\eqref{eq:UIAO-inverter} is precisely the Cauchy completion of the $\V$-category $\X$~\cite{RosebrughWood1994,Stubbe2007a}.
\end{remark}


\paragraph{The ($2$-)monad $\D$.}

The unit of $\D$ is the (fully faithful) Yoneda $\V$-embedding $\y_\X:\X \to \D\X$, $\y_\X (x) = \X(-,x)$, while the multiplication assigns to each $\V$-``downset of downsets'' $\Phi\in \D\X$ its $\V$-valued ``union'': $\mu_\X (\Phi) = \colim{\Phi}{\id_{\D\X}} $, 
which can be rewritten as 
\[
\begin{array}{lll}
\mu_\X (\Phi)
&
= 
&
\bigvee_{\phi\in \D\X} \Phi(\phi)\otimes\phi
\end{array}
\]
In fact, we can do better and rewrite the multiplication as 
\[
\mu_\X (\Phi) = \bigvee_{\phi\in \D\X} \Phi(\phi)\otimes\D\X(\y_\X(-),\phi) = \Phi \otimes (\y_\X)^* = (\D_{-1} \y_\X)(\Phi)
\]
exhibiting thus the adjunction 
\begin{equation}\label{eq:mu-as-inverse-image}
\D\y_X \dashv \mu_\X
\end{equation}


\paragraph{The Kock-Z\"oberlein property.} 
$\D$ is a Kock-Z\"oberlein $2$-monad (abbreviated KZ monad), also known as a lax-idempotent $2$-monad~\cite{Kock95,Zoberlein1976}. 
This is an old well-known result, for which we provide a quick argument below, with the use of the triple adjunction~\eqref{eq:Df:triple-adj} instantiated at the Yoneda embedding: $\D\y_\X \dashv \D_{-1}\y_\X \dashv \D_{\forall} \y_\X$.
From the definition of $\D_\forall$ we can immediately see that 
\begin{equation*}
\D_\forall \y_\X= \y_{\D\X}
\end{equation*} 
Using also~\eqref{eq:mu-as-inverse-image}, we finally obtain
\[
\D\y_\X \dashv 
\mu_\X \dashv 
\y_{\D\X}
\]
exhibiting $\D$ as a KZ-monad. 
For further use, remark that the fully faithfulness of $\y_{\D\X}$ entails the following: 
\begin{enumerate}
\item The counit of $\mu_\X \dashv \y_{\D\X}$ is an isomorphism (equality, in this case, because the codomain is a separated $\V$-category) 
\begin{equation*}
\mu_{\X} \circ \y_{\D\X} = \mathsf{id}_{\D\X}
\end{equation*}
\item The unit of $\D\y_\X \dashv \mu_\X$ is an isomorphism (again, equality)
\[
\mu_\X \circ \D\y_\X = \id_{\D\X}
\]
\end{enumerate}


\paragraph{Monoidal structure of the monad $\D$.}
A commutative monad is a monad in the $2$-category of symmetric monoidal categories and lax symmetric monoidal functors. 
In particular, the unit and the multiplication of the monad are monoidal natural transformation. 
The main result of~\cite{LopezFranco2011} 
is that KZ monads enriched in a $2$-category are (pseudo-)commutative. 
We shall skip the details of showing that $\D$ is $\vcat$-enriched, as this will deviate from our plans. 
We simply point out the monoidal structure of $\D$. 
The reader will easily be able to deduce the missing details himself.

\medskip

The symmetric monoidal structure of $\D$ is given by $\d_0:\one \to \D\one$, the $\V$-functor mapping $0$ to (the constant $\V$-functor to) $e$, and by the family of $\V$-functors $\d_{2,\X,\Y} :\D\X \otimes\D\Y \to \D(\X\otimes\Y)$, natural in $\X$ and $\Y$, obtained as left Kan extensions 
\[
\xymatrix@R=50pt@C=50pt{
\X \otimes\Y 
\ar[r]^-{\y_\X \otimes\y_\Y}
\ar[d]_-{\y_{\X\otimes\Y}}
\ar@{}@<6ex>[d]|-{\to}
&
\D\X \otimes\D \Y
\ar@<1.5ex>@{}[dl]^(.025){}="a"^(.8){}="b" \ar "a";"b"|{\phantom{AA}}
\\
\D(\X \otimes\Y)
&
\ar@{}[u]|{\d_{2,\X,\Y}=\mathsf{Lan}_{\y_\X \otimes\y_\Y}(\y_{\X\otimes\Y})}
}
\]
Explicitly, $\d_{2,\X,\Y}$ maps a pair of presheaves $\phi:\one \nto \X$, $\psi: \one \nto \Y$ to the $2$-variable presheaf 
\[
\d_{2,\X,\Y}(\phi, \psi)(x,y) = \phi (x) \otimes\psi (y)
\]

\begin{proposition}
$\d_{2,\X,\Y}:\D\X \otimes \D\Y \to \D(\X \otimes \Y)$ is a dense $\V$-functor.
\end{proposition}

\begin{proof}
This follows from~\cite[Proposition~5.10]{kelly:book}, using that $\y_{\X \otimes \Y}$ is dense and that $\y_\X \otimes \y_\Y$ is fully faithful. 
\end{proof}

As a consequence of the monad $\D$ being commutative, the Eilenberg-Moore category of $\D$-algebras is (symmetric) monoidal, provided it is conveniently cocomplete~\cite{Jacobs94,Kock1970,Seal2013}. 
This will be the topic of Section~\ref{sec:Vsup}.


\paragraph{$\D$-algebras.} 
Because $\D$ is a KZ $2$-monad, its pseudo-algebras are those $\V$-categories $\A$ for which the unit of the monad (the Yoneda embedding $\y_\A$) has a left adjoint.
That is, the pseudo-$D$-algebras are precisely the cocomplete $\V$-categories, and pseudo-$\D$-morphisms are cocontinuous $\V$-functors.~\cite{Kock95}.\footnote{We shall denote cocomplete $\V$-categories by $\A$, $\B$, etc. to distinguish them from mere $\V$-categories denoted $\X$, $\Y$, and so on.} 

The strict $\D$-algebras are the separated cocomplete $\V$-categories~\cite{Stubbe2017}. 
Observe that pseudo-$\D$-morphisms with codomain separated cocomplete $\V$-categories are always strict, due to the enrichment in a quantale. 

\medskip

Finally, recall that the category of strict $\D$-algebras and strict $\D$-morphisms $\vsup\sep$ is complete and cocomplete as an ordinary category, being monadic over $\set$~\cite{JoyalTierney1984,PedicchioTholen89,Stubbe2006}. 
Just for completeness, we recall again the equational presentation of separated cocomplete $\V$-categories: these are precisely sup-lattices endowed with an action of the quantale $\V$.

\medskip

To ease notation, we shall suppose from now on that all cocomplete $\V$-categories are separated, and simply write $\vsup$ instead of $\vsup\sep$. 
The vigilent reader should remember that $\vsup$ and $\vsup\sep$ are biequivalent $2$-categories~\cite{Stubbe2006}.


\section{The tensor product of cocomplete $\V$-categories}\label{sec:Vsup}

If $\A$ and $\B$ are complete sup-lattices, then $\A \times \B$ is again a complete sup-lattice with pointwise order. 
Here $\A \times \B$ denotes the cartesian product of $\A$ and $\B$ as ordered sets, which happens to be also their tensor product in $\ord =\twocat$. 
It is as well their cartesian product in $\twosup$.

\begin{remark}
For an arbitrary quantale $\V$, it is not true in general that the tensor product in $\vcat$ of two cocomplete $\V$-categories is again cocomplete. 
For example, take $\V$ to be the three-element chain $\{0<a<1\}$, with the non-cartesian idempotent multiplication of Example~\ref{ex:quantales}.\ref{ex:3-quantale}. 
Then $\V$ is cocomplete as a $\V$-category, but the tensor product in $\vcat$ of $\V$ with itself, although cocomplete as a sup-lattice, lacks tensors by $1$. 
Hence it is not a cocomplete $\V$-category.
\end{remark}

The solution to this issue is very familiar to category theorists. 
It essentially relies on the fact that the free cocompletion monad is commutative and that its associated category of (strict) algebras $\vsup\sep$ is cocomplete, hence $\vsup\sep$ carries a monoidal product $-\td-$ which classifies bimorphisms ($\V$-functors which are separately cocontinuous in each argument)~\cite{BanaschewskiNelson1976}. 
This tensor product has been described by various authors~\cite{EklundGutierrez-GarciaHohleKortelainen2018,JoyalTierney1984,Tholen2024}. 
Here we provide another description of $-\td-$, in the spirit of~\cite{Shmuely74}, using specific tools of quantale-enriched categories.


\paragraph{The tensor product as an inverter.}
It is well-known that Eilenberg-Moore algebras for a monad have a canonical presentation via coequalisers of free algebras.  
In the special case of KZ-monads on 2-categories, these coequalisers can be strengthen to coinverters~\cite{LeCreurerMarmolejoVitale02}. 
In particular, the free cocompletion monad $\D$ on $\vcat$ is such a monad, hence each pseudo-$\D$-algebra, that is, each cocomplete $\V$-category, is realised as a {\em reflexive pseudo-coinverter in $\vsup$}:
\begin{equation}\label{eq:D-alg-as-coinverter}
\xymatrix@C=40pt{\D^2 \A \ar@/^1.5ex/[r]^{\mu_\A} \ar@/_1.5ex/[r]_{\D\sup{\A}{}} \ar@{}[r]|{\downarrow} & \D\A \ar[r]^-{ \sup{\A}{}} & \A}
\end{equation}
the common reflection being $\D\y_\A$. 
Actually, more is true:~\eqref{eq:D-alg-as-coinverter} is a {\em split pseudo-coinverter} in $\vcat$, with splittings $\y_\A \vdash \sup{\A}{} $ and $\y_{\D\A}\vdash \mu_{\A}$. 
Of course, if $\A$ is separated, the pseudo-coinverter becomes strict.
In~\eqref{eq:D-alg-as-coinverter} above, the $2$-cell $\mu_\A \le \D \sup{\A}{}$ is obtained from the fully faithful adjoint string
\(
\D\sup{\A}{} \dashv \D\y_\A  \dashv \mu_\A \dashv \y_{\D\A}
\) 
for the cocomplete $\V$-category $\A$. 

\medskip

Let as above $-\td-$ denote the tensor product on $\vsup$ induced by the commutativity of $\D$.\footnote{The notation is consistent with the fact that $\D$-algebras are in fact modules over the quantale $\V$~\cite{PedicchioTholen89,Stubbe2006}.} 
By the general theory of commutative monads, $(\vsup, \td)$ is symmetric monoidal closed, with unit $\V$ and internal hom given by $\vsup(\A, \B)$ (cocompleteness following from cocompleteness of $\B$) and classifying bimorphisms ($\V$-functors which are separately continuous in each argument)~\cite{BanaschewskiNelson1976,Kock72,Seal2013}
\[
\vsup(\A\td \B,\C) \cong \vbisup(\A \otimes\B, \C) \cong \vsup(\A, \vsup(\B,\C))
\]
The usual construction of the tensor product of algebras for a commutative monad gives the tensor product of two cocomplete $\V$-categories as the following reflexive coequalizer in $\vsup$: 
\begin{equation}\label{eq:td-as-coeq}
\xymatrix@C=60pt{
\D(\D\A \otimes\D\B) 
\ar@<+1ex>[r]^{\D(\sup{\A}{}\otimes\sup{\B}{})} \ar@<-1ex>[r]_{\mu_{\A\otimes\B} \circ \d_{2,\D\A,\D\B}} 
&
\D(\A \otimes\B) 
\ar[r]^{}
&
\A \td \B
}
\end{equation}
The above description is often difficult to manipulate (see for example the proof for the associativity of $-\td-$ in~\cite{Seal2013}). 
Here we shall take advantage of the convenient categorical structure of $\vsup$ in order to provide a simpler representation for $-\td-$.

\begin{lemma}\label{lem:refl-coinverter}
The following diagram is a reflexive coinverter in $\vsup$
\[
\xymatrix@C=50pt{
\D^{2}\A \td \D^2 \B 
\ar@<+1.2ex>[r]^{\mu_\A \td \mu_\B} \ar@<-1.2ex>[r]_{\D\sup{\A}{}\td \D\sup{\B}{}} \ar@{}[r]|{\downarrow}
&
\D\A \td \D\B
\ar[r]^-{\sup{\A}{}\td \sup{\B}{}}
&
\A \td \B
}
\]
naturally in the cocomplete $\V$-categories $\A$ and $\B$.
\end{lemma}

\begin{proof}

In the diagram below, each row is a reflexive coinverter in $\vsup$. 
This is because each $\D$-algebra can be obtained as a reflexive coinverter in $\vsup$ of free algebras~\eqref{eq:D-alg-as-coinverter}, and because tensoring $-\td-$ in $\vsup$ preserves colimits in each argument ($\vsup$ being symmetric monoidal closed). 
We can thus apply the 3x3 lemma of~\cite{KellyLack93} to conclude that the diagonal of the 3x3 diagram below is again a reflexive coinverter in $\vsup$:
\[
\raisebox{\dimexpr\depth-5\fboxsep}{\xymatrix@C=50pt@R=50pt{
\D^{2}\A \td \D^2 \B 
\ar@<+1.2ex>[r]^{\mu_\A \td \id} \ar@<-1.2ex>[r]_{\D\sup{\A}{}\td \id} \ar@{}[r]|{\downarrow}
\ar@<+1.2ex>[d]^{\id \td \D\sup{\B}{}} \ar@<-1.2ex>[d]_{\id \td \mu_\B} \ar@{}[d]|{\to}
&
\D\A \td \D^2\B
\ar[r]^{\sup{\A}{}\td \id}
\ar@<+1.2ex>[d]^{\id \td \D\sup{\B}{}} \ar@<-1.2ex>[d]_{\id \td \mu_\B} \ar@{}[d]|{\to}
&
\A \td \D^2\B
\ar@<+1.2ex>[d]^{\id \td \D\sup{\B}{}} \ar@<-1.2ex>[d]_{\id \td \mu_\B} \ar@{}[d]|{\to}
\\
\D^{2}\A \td \D \B 
\ar@<+1.2ex>[r]^{\mu_\A \td \id} \ar@<-1.2ex>[r]_{\D\sup{\A}{}\td \id} \ar@{}[r]|{\downarrow}
\ar[d]_{\id \td \sup{\B}{}}
&
\D\A \td \D\B
\ar[r]^{\sup{\A}{}\td \id}
\ar[d]|{\id \td \sup{\B}{}}
&
\A \td \D\B
\ar[d]^{\id \td \sup{\B}{}}
\\
\D^{2}\A \td \B 
\ar@<+1.2ex>[r]^{\mu_\A \td \id} \ar@<-1.2ex>[r]_{\D\sup{\A}{}\td \id} \ar@{}[r]|{\downarrow}
&
\D\A \td \B
\ar[r]^{\sup{\A}{}\td \id}
&
\A \td \B
}}
\]
\end{proof}

\begin{lemma}\label{lem:3x3}
The following diagram serially commutes:
\[
\xymatrix@C=50pt@R=40pt{
\D^{2}\A \td \D^2 \B 
\ar@<+1.2ex>[r]^{\mu_\A \td \mu_\B} \ar@<-1.2ex>[r]_{\D\sup{\A}{}\td \D\sup{\B}{}} \ar@{}[r]|{\downarrow}
\ar[d]_\cong
\ar@{}[dr]|\cong
&
\D\A \td \D\B
\ar[d]^\cong
\\
\D(\D\A \otimes\D\B) 
\ar@<+1.2ex>[r]^{\D_{-1}(\y_\A \otimes\y_\B)} \ar@<-1.2ex>[r]_{\D(\sup{\A}{}\otimes\sup{\B}{})} \ar@{}[r]|{\downarrow}
&
\D(\A \otimes\B)
}
\]
naturally in the separated cocomplete $\V$-categories $\A$ and $\B$. 
In the above, the unlabelled vertical isomorphisms witness the strong monoidal structure of the free functor $\vcat \to \vsup$. 
\end{lemma}

\begin{proof}
Denote by $\alpha : \D (-) \td \ \D(-) \cong \D(- \ot -)$ the natural isomorphism exhibiting the strong monoidality of the free functor. 
Easy but tedious diagram chasing shows that $\alpha_{\A,\B} \circ (\mu_{\A}\td \mu_\B) = (\mu_{\A \otimes \B} \circ \D\d_2)\circ \alpha_{\D\A,\D\B}$. 
Observing that $\mu_{\A \otimes \B} \circ \D \d_2$ and $\D_{-1}(\y_\A \otimes\y_\B)$ have common right adjoint 
proves $\mu_{\A \otimes \B} \circ \D \d_2\cong \D_{-1}(\y_\A \otimes\y_\B)$ (in fact, this is an equality) and consequently $\alpha_{\A,\B} \circ (\mu_{\A}\td \mu_\B) \cong \D_{-1}(\y_\A \otimes\y_\B)\circ \alpha_{\D\A,\D\B}$.  
The commutativity of the other square is simply given by the naturality of $\alpha$. 
\end{proof}

\begin{corollary}
For every cocomplete $\V$-categories $\A$ and $\B$, $\A \td \B$ is the coinverter in $\vsup$ of the reflexive pair below (with common section $\D(\y_\A \otimes\y_\B)$):
\begin{equation}\label{eq:td-as-coinverter}
\xymatrix@C=60pt{
\D(\D\A \otimes\D\B) 
\ar@<+1.2ex>[r]^{\D_{-1}(\y_\A \otimes\y_\B)} \ar@<-1.2ex>[r]_{\D(\sup{\A}{}\otimes\sup{\B}{})} \ar@{}[r]|{\downarrow}
&
\D(\A \otimes\B)
\ar[r]^q
&
\A \td \B
}
\end{equation}
\end{corollary}

\begin{proof}
This follows from Lemma~\ref{lem:refl-coinverter} and Lemma~\ref{lem:3x3}. 
\end{proof}

We are now ready to provide the promised description of $-\td-$.

\begin{theorem}
For every cocomplete $\V$-categories $\A$ and $\B$, their tensor product $\A\td \B$ in $\vsup$ is the {\em coreflexive inverter} in $\vcat$ below  
\begin{equation}\label{eq:td-as-inverter}
\xymatrix@C=60pt{
\A \td \B \ar[r]^j
& 
\D(\A \otimes\B) 
\ar@<1.5ex>[r]^{\D(\y_\A \otimes\y_\B)}
\ar@{}[r]|{\downarrow}
\ar@<-1.5ex>[r]_{\D_\forall(\y_\A \otimes\y_\B)}
&
\D(\D\A \otimes\D\B) 
}
\end{equation}

\end{theorem}

\begin{proof}
First, observe that both arrows of the parallel pair in~\eqref{eq:td-as-coinverter} have right adjoints, namely 
\(
\D_{-1}(\y_\A \otimes\y_\B) \dashv \D_\forall (\y_\A \otimes\y_\B)
\), 
respectively 
\(
\D(\sup{\A}{} \otimes\sup{\B}{})\dashv \D(\y_\A \otimes\y_\B)
\), and that taking the mate of the inequality $\D_{-1}(\y_\A \ot \y_\B)\le \D(\sup{\A}{}\ot \sup{\B}{})$ produces $\D(\y_\A \otimes\y_\B) \le \D_\forall (\y_\A \otimes\y_\B)$.  
Next, extending the reasoning from~\cite{JoyalTierney1984} to arbitrary quantales, we see that there is a duality ${\vsup}^{\mathsf{op}}
\cong \vsup$, mapping a separated cocomplete $\V$-category $\A$ to the opposite one $\A\op$, and a cocontinuous $f:\A \to \B$ to its right adjoint (see also Section~\ref{sec:Galois}). 
This has the effect of computing colimits in $\vsup$ as limits in $\vcat$; more precisely, in our case, it means that $\A \td \B$ can be explicitly obtained as the coreflexive inverter of the parallel pair of right adjoints in $\vcat$, and that the $\V$-functor $j$ exhibiting the inverter is itself the (necessarily fully faithful) right adjoint of the quotient $\V$-functor $q$ of~\eqref{eq:td-as-coinverter} (see also~\cite{KenneyWood2010} for the case $\V=\two$).   
\end{proof}

\begin{remark}\label{rem:conseq}
We point out some consequences:
\begin{enumerate}
\item By Remark~\ref{rem:inverter}, $\A \td \B$ is a full $\V$-subcategory of $\D(\A \ot \B)$, hence it is separated. 
It has as objects those $\V$-contravariant presheaves $\xi$ on $\A \otimes\B$ satisfying 
\begin{equation}\label{eq:G-ideal}
\D(\A \otimes\B) (\d_{2,\A,\B}(\phi,\psi),\xi) = \xi (\sup{\A}{\phi},\sup{\B}{\psi})
\end{equation}
for all $\phi\in \D\A, \psi\in \D\B$. 
Explicitly, this means 
\begin{equation}\label{eq:G-ideal-1}
\bigwedge_{(a,b)} [\phi(a)\ot\psi(b),\xi(a,b)] = \xi (\sup{\A}{\phi},\sup{\B}{\psi})
\end{equation}
Readers familiar with order and lattice theory will recognise in the above the quantale-enriched version of the notion of $G$-ideal \cite{Shmuely74}. 
It may be instructive to compare it with the description of the tensor product of cocomplete $\V$-categories in \cite{JoyalTierney1984}, and in the more recent~\cite{EklundGutierrez-GarciaHohleKortelainen2018,Tholen2024}. 

\item The cocomplete $\V$-category $\A \td \B$ is not only a full $\V$-subcategory of $\D(\A \ot \B)$, but also reflective in $\D(\A \ot \B)$; consequently, limits in $\A \td \B$ are computed as in $\D(\A\otimes\B)$ (pointwise), while colimits are obtained by computing them in $\D(\A \otimes\B)$ and then applying the reflector $q$.

\item Any contravariant presheaf is canonically a colimit of representables. In particular, any $\xi \in \A \td \B$ can be written as  
\[
\xi 
= 
q(\xi) 
=
q\left(\bigvee_{(a,b)} \xi(a,b) \ot \y_{\A\ot \B}(a,b)\right)
=
\bigvee_{(a,b)} \xi(a,b) \ot (q \circ\y_{\A\ot \B})(a,b)
\]
In the above, the first join and tensor are computed in $\D( \A \ot \B)$, while the last ones are in $\A \td \B$.
This has the advantage as exhibiting the objects of $\A \td \B$ as weighted colimits of representables $q(\A(-,a)\ot \B(-,b))$. These are the quantale-enriched versions of the maps $L_b^a$ appearing in~\cite{Shmuely74}.

\item To gain some insight on objects of $\A \td \B$, let $\phi$ to be $\A(-,a)$ for some $a$ in $\A$, and $\psi = \bigvee_i \B(-,b_i)$ for some family of objects $(b_i)_i$ of $\B$. 
For an object $\xi$ in $\A \td \B$, relation~\eqref{eq:G-ideal-1} becomes $\bigwedge_i \xi(a,b_i)=\xi(a,\bigvee_i b_i)$. 
Similarly, $\bigwedge_j \xi(a_j,b)=\xi(\bigvee_j a_j, b)$ holds for any family $(a_j)_j$ of objects of $\A$ and any object $b$ of $\B$. 
Taking the index set to be empty, we obtain $\xi(a, \bot_\B) = \top = \xi(\bot_\A, b)$ for all $a,b$, where $\bot_\A$ denotes the least object in the underlying order of $\A$, respectively $\top_\B$ is the greatest object of $\B$, and $\top$ is the greatest element of the quantale $\V$ (absolute truth, if we interpret the elements of the quantale as logical values). 
Compare these relations with those in the definition of a G-ideal~\cite{Shmuely74}.

\end{enumerate}

\end{remark}


\paragraph{The tensor product classifies bimorphisms.} 
Although this follows from the general theory of commutative \linebreak algebraic theories (and monads), here we provide a direct proof. 
Let $i$ be the composite $\V$-functor 
\[
\xymatrix{\A \otimes\B \ar[r]^-{\y} & \D(\A \otimes\B) \ar[r]^{q} & \A \td \B}
\]

\begin{lemma}\label{lem:bimorf}
The following diagram commutes (strictly, as $\A \td \B$ is separated):
\[
\xymatrix{
\D\A \otimes\D\B 
\ar[r]^{\d_{2,\A,\B}} 
\ar[d]_{\sup{\A}{} \otimes\sup{\B}{}}
& 
\D(\A \otimes\B) 
\ar[d]^q
\\
\A \otimes\B 
\ar[r]^{i} 
& 
\A \td \B
}
\]
\end{lemma}

\begin{proof}
The inequality 
\[
i \circ (\sup{\A}{}\otimes\sup{\B}{}) 
\ge q \circ \d_{2,\A,\B}
\]
is the mate of the $2$-cell (inequality) in the diagram below: 
\[
\xymatrix{
\A \otimes\B \ar[r]^-{\y_{\A\otimes \B}}  \ar[d]_{\y_\A\ot\y_\B} 
\ar `u []+<5em,0cm>`[rr]^{i}[rr]
& \D(\A \otimes\B)\ar@{=}[dr] \ar[r]^{q} & \A \td\B \ar[d]^{j} 
\\
\D\A\ot\D\B \ar[rr]^{\d_{2,\A,\B}}  & \ar@{}[ur]|>>>>>{\nearrow} & \D(\A\ot\B)
}
\]
For the opposite inequality, notice that
\[
\begin{array}{lllll}
e 
&\le & 
\D(\A \ot \B)(\d_{2,\A,\B}(\phi, \psi), \d_{2,\A,\B}(\phi, \psi))
\\[10pt]
&\le &
\D(\A \ot \B) (\d_{2,\A,\B}(\phi, \psi), (q \circ \d_{2,\A,\B})(\phi, \psi))
\\[10pt]
&=& 
(q \circ \d_{2,\A,\B})(\phi, \psi)(\sup{\A}{\phi}, \sup{\B}{\psi})
\\[10pt]
&=&
\D(\A \ot \B)(\d_{2,\A,\B} \circ (\y_\A \ot \y_\B) \circ ( \sup{\A}{} \ot \sup{\B}{} )(\phi, \psi), (q \circ \d_{2,\A,\B})(\phi, \psi))
\end{array}
\]
holds for any $\phi\in \D\A$ and $\psi\in \D\B$.
Therefore $\d_{2,\A,\B} \circ (\y_\A \ot \y_\B)\circ (\sup{\A}{}\ot \sup{\B}{}) \le q \circ \d_{2,\A,\B}$ holds, from which $i \circ (\sup{\A}{}\otimes\sup{\B}{})  = q \circ \y_{\A \ot \B}\circ (\sup{\A}{}\ot \sup{\B}{}) \le q \circ \d_{2,\A,\B}$ follows. 

\end{proof}

\begin{lemma}\label{lem:bimorf2}
The $\V$-functor $i:\A \ot \B \to \A \td  \B$ is a bimorphism. 
\end{lemma}

\begin{proof}
In the diagram below, 
\[
\xymatrix{\D\A \otimes\D\B \ar[r]^{\d_{2,\A,\B}}\ar[d]_{\sup{\A}{}\otimes\sup{\B}{}}
&
\D(\A \otimes\B) \ar[r]^{\D i} \ar[d]^{q} & \D(\A \td \B) \ar[d]^{\sup{\A \td \B}{}}
\\
\A \otimes\B \ar[r]^{i}
&
\A \td \B \ar@{=}[r]
&
\A \td \B}
\]
the left square commutes by Lemma~\ref{lem:bimorf}, while the commutativity of the right square holds because $q$ is a morphism in $\vsup$. 
\end{proof}

\begin{proposition}
The $\V$-functor $i:\A \ot \B \to \A \td  \B$ is dense and point-separating with respect to the the forgetful functor $\vsup \to \vcat$.
\end{proposition}

\begin{proof}
Because $\Lan_{\y_{\A \ot \B}} (q\circ \y_{\A \ot \B}) $ certainly exists, with identity unit $q\circ \y_{\A \ot \B} = \Lan_{\y_{\A \ot \B}} (q\circ \y_{\A \ot \B}) \circ \y_{\A \ot \B}$, the theory of iterated left Kan extensions gives
\[
\begin{array}{lllll}
\Lan_{i}(i) 
&=& 
\Lan_{q\circ \y_{\A \ot \B}}(q\circ \y_{\A \ot \B}) 
&=& \Lan_q \Lan_{\y_{\A \ot \B}} (q\circ \y_{\A \ot \B}) 
\\
&=& 
\Lan_q (q \circ \Lan_{\y_{\A \ot \B}} (\y_{\A \ot \B})) 
&=&
\Lan_q(q) 
\\
&=& q \circ j &=& \id_{\A \td \B}
\end{array}
\]
using that the Yoneda embedding $\y_{\A \otimes \B}$ is dense and that $q$ has a fully faithful right adjoint. 
Hence $i$ is dense.
For second statement, let $f,g:\A\td \B \to \C$ be morphisms in $\vsup$ such that $f \circ i = g \circ i$ holds. 
Then $f \circ q$ and $g \circ q$ are cocontinuous $\V$-functors which coincide on the representables. 
Consequently, $f\circ q = g \circ q$, which in turn implies $f = f\circ q\circ j = g\circ q\circ j = g$. 
\end{proof}

\begin{theorem}
The $\V$-functor $i:\A \ot \B \to \A \td  \B$ is the universal bimorphism.
\end{theorem}

\begin{proof}
We will show that 
\[
\vsup(\A \td \B,\C) \cong \vbisup(\A \otimes \B, \C)
\]
holds, naturally in the cocomplete $\V$-categories $\A,\B,\C$. 

First, by general abstract nonsense, given a cocontinuous $\V$-functor $f:\A\td \B \to \C$, precomposition with $i$ yields a bimorphism $f\circ i$. 

Conversely, let $g:\A \otimes \B \to \C$ be a bimorphism. Then $g$ uniquely extends to a cocontinuous $\V$-functor $\Lan_{\y_{\A \otimes \B}}(g):\D(\A \otimes \B) \to \C$. 
In particular (being cocontinuous), $\Lan_{\y_{\A \otimes \B}}(g)$ has a right adjoint, given by $\Lan_g (\y_{_{\A \otimes \B}}) \cong \C(g(-),-)$. 
Taking a second left Kan extension produces 
\[
\Lan_i (g) \cong \Lan_q \left(\Lan_{\y_{_{\A \otimes \B}}}(g)\right) \cong \Lan_{\y_{_{\A \otimes \B}}}(g) \circ j 
\]
Now $g$ factorises through $i$ if and only if $\Lan_{\y_{\A \ot \B}}(g)$ factorises through $q$, if and only if its right adjoint factorises through the inverter $j$. 
That is, if the image of $\C(g(-),-)$ is in $\A \td \B$. 
To see this, consider $c\in \C$ and $\phi\in \D\A, \psi \in \D\B$. 
We have:
\begin{align*}
& \C(g(\sup{\A}{\phi},\sup{\B}{\psi}),c) 
&&
= 
&&
\C (\sup{\C}{}\circ \D g \circ \d_2(\phi, \psi),c)
&&
\\
&&&
=
&&
\D \C(\D g \circ \d_2(\phi, \psi),\C(-,c))
&&
\\
&&&
=
&&
\D(\A \otimes \B)(\d_2(\phi, \psi), \D_{-1} g \circ \y_\C(c))
&&
\\
&&&
=
&&
\D(\A \otimes \B)(\d_2(\phi, \psi), \C(g(-,-),c)))
\end{align*}
where in the first line we used that $g$ is a bimorphism. 
\end{proof}

\begin{remark}
Unravelling the details of the above proof, we see that the inverse correspondences are given by: 
\[
\begin{array}{ccc}
    \vsup(\A \td \B,\C) 
    &
    \cong 
    &
    \vbisup(\A \otimes \B, \C)
    \\
    \mbox{%
    	$%
			\xymatrix@C=13pt{%
				\A\td \B 
				\ar[r]^-f 
				& 
				\C
			}
		$
    } 
    &
    \mapsto
    &
    \mbox{%
    	$%
			\xymatrix@C=13pt{%
				A \otimes \B 
				\ar[r]^-{i}
				& 
				\A\td \B \ar[r]^-f 
				& 
				\C
			}
		$
	}
	\\
	\mbox{%
    	$%
			\xymatrix@C=30pt{%
				\A \td \B 
				\ar[r]^-{j}
				& 
				\D(\A \otimes \B) 
				\ar[r]^-{\Lan_{\y_{\A \otimes \B}}(g)} 
				& 
				\C
			}
		$
	}	
	&
	\mapsfrom
	&
	\mbox{$\xymatrix@C=13pt{\A \otimes \B \ar[r]^-g & \C}$}
\end{array}
\]
\end{remark}


\paragraph{Tensor product of cocomplete $\V$-categories via Galois maps.}\label{sec:Galois}

Before ending this section, we shall discuss the link between the previous results and the theory of Galois connections~\cite{Mowat1968,Shmuely74,BanaschewskiNelson1976}.
It has been observed in~\cite[Section 3.1.4]{EklundGutierrez-GarciaHohleKortelainen2018} that $\vsup$ is a $*$-autonomous category, with internal hom representing bimorphisms (see also~\cite{Tholen2024}). 

\medskip

First, as it is the case for complete sup-lattices~\cite[Section II.1, Proposition~1]{JoyalTierney1984}, there is a duality of $2$-categories~\cite[Section II.1, Proposition~2]{JoyalTierney1984},~\cite[Section 2]{Stubbe2007a}
\begin{equation}\label{eq:duality}
 \vsup  \cong  {\vsup\op}
\end{equation}
sending a cocomplete $\V$-category $\A$ to $\A\op$, a cocontinuous $\V$-functor $f:\A \to \B$ to $g\op:\B\op \to 
\A\op$, where $g$ is the right adjoint of $f$,%
\footnote{The right adjoint $g$ exists because $f$ is cocontinuous and $\A$ is cocomplete, being given by $g=\Lan_f \id_\B$.
} %
and a $2$-cell $f\le_\B f' $ to ${g'} \le_\A g$, hence to $g \le_{\A\op} g'$. 
This has in particular the advantage of expressing colimits in $\vsup$ as limits in $\vsup$, which are in turn formed in $\vcat$ (pointwise) by monadicity of the forgetful functor. 

\medskip

Recall that $\vsup$ is symmetric monoidal closed with unit the cocomplete $\V$-category $\V$ and internal hom $\vsup(\A,\B)$, where the (pointwise) $\V$-category structure is inherited from $\vcat(\A,\B)$.\footnote{``In mathematical experience, closed structure appears more canonical than monoidal structure: we understand the vector space of linear maps more readily than the tensor product of vector spaces.''~\cite{HylandPower02}} %
In particular, $\A\cong \vsup(\V,\A)$ holds by the general theory of monoidal closed categories. 
Applying the duality~\eqref{eq:duality}, we obtain
\[
\A\op \cong \vsup(\V,\A\op)\cong \vsup(\A, \V\op)
\]
%
Unraveling the isomorphisms, we can see that the latter correspondence is given by (taking the opposite of) the contravariant Yoneda $a \in \A\op \mapsto 
g_a = \A(-,a):\A \to \V\op$, with inverse
\begin{align*}
&
g\in \vsup(\A,\V\op) \ 
&&
\mapsto
&&
a_g\in \A\op, 
a_g = \bigvee_x g(x)\otimes x
\end{align*} 
where the join and the tensors are computed in $\A$. 
In particular, notice that $e\le g(a_g)$ and $g(a_g) \otimes a_g\le a_g$ always hold, which in turn imply 
$a_g = %
g(a_g) \otimes a_g %
$. These relations, together with those of Remark~\ref{rem:conseq}, will be deferred for future investigation.

From a more categorical perspective, observe that all the above imply that $\vsup$ is not only symmetric monoidal closed, but in fact $*$-autonomous, with dualising object $\V\op$, because
\begin{align*}
\vsup(\vsup(\A,\V\op),\V\op) \cong \vsup(\A\op,\V\op) \cong \vsup(\V,\A) \cong \A
\end{align*}
By general $*$-autonomous category theory, it follows that the tensor product of cocomplete $\V$-categories can be expressed in terms of the internal hom as 
\[
\A \td \B \cong \vsup(\A,\B\op)\op
\]
One can also directly check the isomorphism above, using the explicit description of $\A\td B$ obtained in this section:

\begin{proposition}
The correspondences below induce an isomorphism of cocomplete $\V$-categories, between \linebreak $\vsup(\A,\B\op)\op$ (the $\V$-category of left adjoint maps) and the tensor product $\A\td \B$ (whose objects are the analogue of Shmuely's G-ideals)
\[
\xymatrix@R=5pt{
\vsup(\A,\B\op)\op \ar@{}[r]|-{\cong} & \A\td \B
\\
f \ar@{}[r]|-{\mapsto} & \xi=\B(-,f-)
\\
f=\bigvee_{b\in \B} \xi(-,b)\otimes b & \xi \ar@{}[l]|-{\mapsfrom}
}
\]
\end{proposition}

\begin{proof} 
It is easy to see that the above correspondences are indeed well defined and inverses to each other. 
\end{proof} 


\section{Nuclearity in $\vsup$}\label{sec:Vccd}


\paragraph{Nuclearity in symmetric monoidal closed categories.} 
Grothendieck introduced in Functional Analysis the concept of nuclearity for objects and morphisms, in order to reproduce {finite dimensionality} behaviour (for objects) and matrix calculus (for arrows) \cite{Grothendieck1955}.
At some point later, it has been observed that the concept of nuclearity is, in fact, essentially categorical in nature~\cite{Rowe1988}: %
In a symmetric monoidal closed category, an arrow $f:\A \to \B$ is called {\em nuclear} if the associated $\one \to 
[\A,\B]$\footnote{Sometimes called the {\em name} of $f$.} factorises through the canonical arrow $\B \otimes \A^*\to[\A, \B]$, where $\one$ is the unit for the tensor product and $\A^{*}=[\A,\one]$: 
\[
\xymatrix@C=50pt{
\one 
\ar@{.>}[r]
& 
\B \otimes \A^* 
\ar[r] 
& 
[\A,\B]
\ar@{<-}`u[ll]`[ll]
}
\]
An object $\A$ is called {\em nuclear} if $\id_{\A}$ is so~\cite{HiggsRowe1989}. 
Nuclear objects are also called {\em dualizable}.
Trivially, the unit $\one$ is always nuclear. 
We recall below some results on nuclearity of objects. References are~\cite{AbramskyBlutePanangaden1999,HiggsRowe1989,kelly1972-lnm281,KellyLaplaza1980,Rowe1988}:

\begin{lemma}
In a symmetric monoidal closed category, an object $\A$ is nuclear if and only if one of the following equivalent conditions holds:

\begin{enumerate}

\item The canonical morphism $\B \otimes [\A, \C] \to [\A, \B \otimes \C]$ is an isomorphism for all $\B$ and $\C$.

\item The canonical morphism $\B \otimes \A^*\to [\A,\B]$ is an isomorphism for all $\B$.

\item The canonical map $\A\otimes \A^* \to [\A,\A]$ is an isomorphism.

\item $\A$ has a right dual $\A^*$ (necessarily isomorphic to $[\A, \one]$), with unit $\one \to \A \otimes \A^*$ and  counit $\A^* \otimes \A \to \one$ satisfying the usual triangle identites. 
\end{enumerate}

\end{lemma}

\begin{lemma}\label{lem:properties-nuclear}
In a symmetric monoidal closed category, if $\A, \B$ are nuclear, then so are $\A \otimes \B$ and $[\A,\B]$ (hence also $\A^*$). Retracts of nuclear objects are again nuclear. 
\end{lemma}

\begin{proposition}
For any symmetric monoidal closed category, the full subcategory of nuclear objects is compact closed.
\end{proposition}

\begin{example}

\begin{enumerate}

\item If $R$ is a commutative ring, then the nuclear objects in the category of $R$-modules are the finitely-generated projective ones~\cite{Rowe1988}.

\item In a cartesian closed category, only the terminal object is nuclear. 

\item For an associative algebra $A$ over a commutative field $\Bbbk$, the category of right $A$-modules is {nuclear} in the (2-) category $LocPres_\Bbbk$ of locally presentable $\Bbbk$-linear categories and cocontinuous $\Bbbk$-linear functors, with respect to the Kelly-Deligne tensor product~\cite{BrandenburgChirvasituJohnsonFreyd2015}.

\item Let $C$ be a coassociative coalgebra over a commutative field $\Bbbk$. 
Then the category of right $C$-comodules is {nuclear} in $LocPres_\Bbbk$ if and only if it has {enough projectives} ($C$ is right semiperfect)~\cite{BrandenburgChirvasituJohnsonFreyd2015}, as it is the case, for example, when $C$ is finite dimensional. 

\item In the category of Banach spaces and bounded linear maps, the nuclear objects are the finite-dimensional Banach spaces~\cite{AbramskyBlutePanangaden1999,Rowe1988}.

\item The nuclear objects in the category of complete sup-lattices are the completely distributive lattices~\cite{HiggsRowe1989}. 

\end{enumerate}

\end{example}


\paragraph{Completely distributive cocomplete $\V$-categories.}

There are many equivalent formulations of the notion of  completely distributive cocomplete $\V$-category. For us, the most appropriate will be the following~\cite{RosebrughWood1994}:

\begin{definition}
A {\em completely distributive} $\V$-category%
\footnote{
Also know as a {\em totally continuous} $\V$-category.%
} %
 is a cocomplete $\V$-category $\A$ such that $\sup{\A}{}:\D\A\to \A$ has a $\V$-enriched left adjoint $\t_\A:\A \to \D\A$. 
\end{definition}

Completely distributive $\V$-categories have been studied in the past; main references are~\cite{LaiZhang06,PuZhang15,Stubbe2007a}. 
As the name suggests, for $\V=\two$ we recover the well-known completely distributive lattices (the left adjoint to $\sup{\A}{}$ mapping an element to the downset of those totally below it).  
We shall denote by $\vccd$ the category of completely distributive $\V$-categories and continuous and cocontinuous $\V$-functors, and by $\vccd_{\mathsf{sup}}$ the category of completely distributive $\V$-categories with cocontinuous $\V$-functors. 
Recall our earlier convention that all (cocomplete) $\V$-categories are assumed to be separated.

\begin{example}
For any $\V$-category $\A$, $\D\A$ is completely distributive, the left adjoint of $\sup{\D\A}{}$ being $\D\y_\A $~\cite{LaiZhang06,Stubbe2007a}.
In particular, for any set $X$, the $\V$-valued powerset $\V^X = [\mathsf{d}X,\V]$ is completely distributive. 
Taking $X$ to be a singleton shows that the quantale $\V$ is itself {\em completely distributive as a $\V$-category}, 
while $X=\emptyset$ produces the completely distributive terminal $\V$-category $\mathbb 1_\top$. 
Also, the unit $\V$-category $\mathbb 1$ is completely distributive.
\end{example}

\begin{remark}\label{rem:proj}
\begin{enumerate}

\item A cocomplete $\V$-category is completely distributive if and only if it is a projective object in $\vsup$~\cite{Stubbe2007a} with respect to the class of all epimorphisms. Epimorphisms in $\vsup$ are precisely the cocontinuous $\V$-functors which are epis in $\vcat$. They can alternatively be described as $\V$-functors with fully faithful right adjoints.

\item Complete distributivity of a cocomplete $\V$-category does not necessarily entail the complete distributivity of the underlying lattice. 
For example, $\V$ itself is always completely distributive as an $\V$-category~\cite{Stubbe2007a}, but not necessarily distributive as a lattice. 
However, there exists a positive result in this sense, due to~\cite{LaiZhang06}: every completely distributive $\V$-category $\A$ is completely distributive as a lattice if and only if $\V$ itself is a completely distributive lattice.
Actually, if the reader is interested in cocomplete $\V$-categories which are not completely distributive, there is a simple way of producing such examples: taking any complete sup-lattice $\A$ which {\em is not completely distributive}, like the diamond lattice $M_{3}$, and a quantale $\V$ which {\em is completely  distributive} as a lattice.
Then the tensor product of $\A$ and $\V$ in the category of complete sup-lattices is a cocomplete $\V$-category -- it is the free cocomplete $\V$-category over the complete sup-lattice $\A$~\cite{JoyalTierney1984}, but not completely distributive as a lattice~\cite{Shmuely1979}. 
Therefore this tensor product is neither $\V$-completely distributive.   
\end{enumerate}

\end{remark}


\begin{theorem}\label{thm:nuclear=ccd}
Let $\A$ be a cocomplete $\V$-category. The following are equivalent:
\begin{enumerate}
\item \label{vccd} $\A$ is a completely distributive $\V$-category.

\item \label{projective} $\A$ is a projective object in $\vsup$ with respect to (regular) epimorphisms.

\item \label{nuclear} $\A$ is a nuclear object in the $*$-autonomous category $\vsup$.

\end{enumerate}
\end{theorem}

\begin{proof}
The proof will be split into several parts. First, the equivalence~\ref{vccd}$\Longleftrightarrow $\ref{projective} is due to~\cite{Stubbe2007a}, as mentioned earlied in Remark~\ref{rem:proj}. 
Next, the implication~\ref{projective}$\Longrightarrow $\ref{nuclear} will follow from the next Lemma~\ref{lem:free},~\cite[Proposition~3.3]{Stubbe2007a} and~\cite[Proposition~1.4]{Rowe1988}. 
Finally,~\ref{nuclear}$\Longrightarrow $\ref{projective} is the subsequent Proposition~\ref{prop:nuclear}. 
\end{proof}

\begin{lemma}\label{lem:free}
Free cocomplete $\V$-categories are nuclear in $\vsup$.
\end{lemma}

\begin{proof} 
If the cocomplete $\V$-category $\A$ is $\D\X$ for a $\V$-category $\X$, then its dual is $\A^* = \vsup(\A,\V) =\vsup(\D\X,\V)\cong \vcat(\X,\V) = \D(\X\op)$. The coevaluation map
\[
\V \to \A \td \A^* \cong \D\X \td \D(\X\op) \cong \D(\X \otimes\X\op)
\]
is provided by the $\V$-hom $\X \op \otimes\X \to \V$. 
The evaluation map 
\[
\A^* \td \A \cong \D(\X\op) \td \D\X \cong\D(\X\op \otimes \X) \to \V
\]
is the left Kan extension of the $\V$-category structure on $\X\op$ along the Yoneda embedding. It is easy to see that the usual triangle identities are satisfied, hence $\D\X$ is nuclear, with dual $\D(\X\op)$. 
\end{proof}

\begin{proposition}\label{prop:nuclear}
Any nuclear object of $\vsup$ is projective. 
\end{proposition}

\begin{proof}
That nuclear implies projective in $\vsup$ can be seen as in~\cite{RosebrughWood1994}, using the description of (regular) epimorphisms in $\vsup$, and the equivalence between projective objects in $\vsup$ and $\vccd$, both due to~\cite{Stubbe2007a}. 
Although no originality is claimed here, we include below the details for completeness.
Let $\A$ in $\vsup$ be a nuclear object. 
Then $\B \td \A^* \cong \vsup(\A,\B)$ holds for all $\B$ in $\vsup$. 
Consider $f:\A \to \B$ and $e:\C \to \B$ a (regular) epimorphism in $\vsup$ (a dense cocontinuous $\V$-functor).
As the tensor $-\td -$ preserves colimits and epimorphisms coincide with regular epimorphisms in $\vsup$, $e \td \A^*: \C \td \A^* \to \B \td \A^*$ is again epi, hence $g=\vsup(\A,e):\vsup(\A,\C) \to \vsup(\A,\B)$ is epimorphism in $\vsup$. 
Let $h$ be the (fully faithful) right adjoint of $g$; then $f:\A \to \B$ factorises through $e:\C \to \B$ via $h(f)$.
Now~\cite[Lemma~3.4]{Stubbe2007a} shows that this is indeed the desired factorisation. 
\end{proof}


\paragraph{Tensor product of completely distributive $\V$-categories in $\vsup$.} 
Lemma~\ref{lem:properties-nuclear} and Theorem~\ref{thm:nuclear=ccd} immediately show that the tensor product of two completely distributive cocomplete $\V$-categories is again completely distributive. 
However, we believe it is instructive to see also a direct proof generalising~\cite{KenneyWood2010}, using the explicit description of $-\td-$ obtained in the previous section. 

\medskip

Recall that for a completely distributive $\V$-category $\A$ we denoted by $\t_\A:\A \to \D\A$ the left adjoint of $\sup{\A}{}$. 
This is the $\V$-valued analogue of the ``totally below downset''. 
Denote by $\Downarrow:\A \nto\A$ the corresponding distributor, $\Downarrow(a,a') = \t_\A(a')(a)$~\cite{Stubbe2007a}.

\begin{proposition}
For $\A$, $\B$ completely distributive $\V$-categories, the coinverter arrow $q:\D(\A \otimes\B) \to \A \td \B$ of~\eqref{eq:D-alg-as-coinverter} is 
\[
q(\xi)(a,b) = \D(\A \otimes\B)(\d_2(\t_\A(a), \t_\B(b)),\xi)
\]
\end{proposition}

\begin{proof}
Denote for now by $q_1$ the $\V$-functor given by the expression above. 
We will show that the image of $q_1$ lives in $\A \td \B$ and that $q_1$ indeed provides the left adjoint to the embedding $j:\A\td \B \to \D(\A \otimes\B)$, being thus isomorphic to $q$.

First, observe that for any $\phi\in \D\A$, 
\[
\begin{array}{lllll}
\Downarrow_\A (-,\sup{\A}{\phi}) 
&=&
\Downarrow_\A(-,-)\otimes \Downarrow_\A(-,\sup{\A}{\phi})
&\le &
\Downarrow_\A (-,-)\otimes \phi
\end{array}
\]
holds, with the last inequality being a consequence of the adjunction $\t_\A \dashv \sup{\A}{}$. 

Now, for any $\phi\in \D\A$, $\psi\in \D\B$ we have
\[
\begin{array}{lllll}
\D(\A \otimes\B)(\d_2(\phi, \psi), q_1(\xi)) 
&=&
\bigwedge_{a,b} [\phi(a)\otimes\psi(b), \bigwedge_{x,y} [\Downarrow_\A(x,a) \otimes\Downarrow_\B(y,b),\xi(x,y)]]
\\[10pt]
& = & 
\bigwedge_{x,y} [ \bigvee_a \phi(a) \otimes\Downarrow_\A(x,a)\otimes\bigvee_b \psi(b) \otimes\Downarrow_\B(y,b),\xi(x,y)]
\\[10pt]
&\le &
\bigwedge_{x,y} [ \Downarrow_\A(x,\sup{\A}{\phi})\otimes\Downarrow_\B(y,\sup{\B}{\psi}),\xi(x,y)]
\\[10pt]
&=&
q_1(\xi)(\sup{\A}{\phi},\sup{\B}{\psi})
\end{array}
\]
Therefore, by~\eqref{eq:G-ideal}, $q(\xi)$ is in $\A \td\B$. 

Next, observe that $\id_{\D(\A \otimes\B)}\le q_1$ holds:
\[
\begin{array}{lllll}
\xi 
&=& 
\D(\A \otimes\B)(\y_{\A \otimes\B}(-,-),\xi) 
&=&
\D(\A \otimes\B)(d_2(\y_\A,\y_\B)(-,-),\xi)
\\[10pt]
&\le&
\D(\A \otimes\B)(d_2(\t_\A,\t_\B)(-,-),\xi)
&=&
q_1(\xi)
\end{array}
\] 
In particular, if $\xi\in \A\td \B$ and $\theta \in \D(\A \otimes\B)$ satisfy $q_1(\theta)\le \xi$, then also $\theta \le q_1(\theta) \le \xi=j(\xi)$ holds in $\D(\A \otimes\B)$.

Finally, consider again $\xi\in \A\td \B$ and $\theta \in \D(\A \otimes\B)$, but such that $\theta \le \xi$ in $\D(\A \otimes\B)$. 
Then 
\[
\begin{array}{lllllll}
q_1(\theta) 
&=& 
\D(\A \otimes\B)(\d_2(\t_\A(-), \t_\B(-)),\theta) 
&\le &
\D(\A \otimes\B)(\d_2(\t_\A(-), \t_\B(-)),\xi) 
\\[10pt]
&\le &
\xi(\sup{\A}{\t_\A(-)}, \sup{\B}{\t_\B(-)})
&=&
\xi
\end{array}
\] 
using that $\A$ and $\B$ are completely distributive. 
\end{proof}

\begin{corollary}
The monoidal structure of $\vsup$ restricts to $\vccd_{sup}$.
\end{corollary}

\begin{proof}
The unit for $\td$ is $\D\one =\V$, which is completely distributive by~\cite{LaiZhang06}. 
That the tensor product of two completely distributive lattices $\A$ and $\B$ is again completely distributive can now be seen using~\cite{RosebrughWood1995}: the fully faithful adjoint string 
\[
\D(\t_\A \otimes\t_\B) \dashv \D(\sup{\A}{}\otimes\sup{\B}{}) \dashv \D(\y_\A \otimes\y_\B) \dashv \D_{-1}(\y_\A \otimes\y_\B) \dashv \D_\forall (\y_\A \otimes\y_\B)
\]
induces another fully faithful adjoint string in $\vcat$ involving the inverter of $\D(\t_\A \otimes\t_\B) \le \D(\y_\A \otimes\y_\B)$, that we shall denote by $\iota$, and the inverter of $\D(\y_\A \otimes\y_\B) \le \D_\forall (\y_\A \otimes\y_\B)$, namely the embedding $j:\A\td \B \to \D(\A\otimes\B)$ described by~\eqref{eq:td-as-inverter}. 
More precisely, this newly adjoint string writes as $\iota\dashv q\dashv j$, with $q$ as earlier, both left and (now) right adjoint:
\[
\xymatrix@C=60pt{
\A \td \B 
\ar@<+3.5ex>[r]^\iota
\ar@{}@<2ex>[r]|\perp
\ar@{<-}[r]|q
\ar@{}@<-2ex>[r]|\perp
\ar@<-3.5ex>[r]_j
& 
\D(\A \otimes\B) 
\ar@<6.3ex>[rr]^{\D(\t_\A \otimes\t_\B)}
\ar@{}@<5ex>[rr]|\perp
\ar@<3.2ex>@{<-}[rr]|{\D(\sup{\A}{}\otimes\sup{\B}{})}
\ar@{}@<+1.9ex>[rr]|\perp
\ar[rr]|{\D(\y_\A \otimes\y_\B)}
\ar@{}@<-1.9ex>[rr]|\perp
\ar@<-3.2ex>@{<-}[rr]|{\D_{-1}(\y_\A \otimes\y_\B)}
\ar@{}@<-5ex>[rr]|\perp
\ar@<-6.3ex>[rr]_{\D_\forall(\y_\A \otimes\y_\B)}
&&
\D(\D\A \otimes\D\B) 
}
\]
This exhibits $\A \td \B$ as a retract in $\vsup$ of the completely distributive $\D(\A\otimes\B)$, hence it is itself completely distributive. Therefore the monoidal structure of $\vsup$ restricts to $\vccd_{sup}$ turning it into a compact closed category. 
\end{proof}


\section*{Acknowledgments}
The author would like to thank the organisers of the conference {\sf Category Theory at Work in Computational Mathematics and Theoretical Informatics CATMI 2023} for the invitation to present her work and to submit it to the conference proceedings volume. Thanks go also to the anonymous referees for several useful suggestions on the first version of this paper.

%
%
\bibliographystyle{alpha}
\bibliography{ABbib}


\end{document}